\documentclass[11pt]{amsart}

\usepackage{amssymb}
\usepackage{mathrsfs}
\usepackage{amsfonts}
\usepackage{latexsym,amsmath,amsthm,amssymb,amsfonts}
\usepackage[usenames]{color}
\usepackage{xspace,colortbl}
\usepackage{graphicx}
\usepackage{tipa}
\usepackage{amsthm}

\newcommand{\be}{\begin{equation}}
\newcommand{\ee}{\end{equation}}
\newcommand{\beq}{\begin{eqnarray}}
\newcommand{\eeq}{\end{eqnarray}}

\newtheorem*{theoremA}{Theorem A}

\def\begeq{\begin{equation}}
\def\endeq{\end{equation}}

\def\odot{\setbox0=\hbox{$\bigcirc$}\relax \mathbin {\hbox
to0pt{\raise.5pt\hbox to\wd0{\hfil $\wedge$\hfil}\hss}\box0 }}

\numberwithin{equation} {section}

\numberwithin{equation}{section}
\newtheorem{theorem}{\bf Theorem}[section]
\newtheorem{proposition}[theorem]{\bf Proposition}

\newtheorem{lemma}[theorem]{\bf Lemma}

\allowdisplaybreaks

\begin{document}
\title[Capillary Orlicz-Minkowski flow]
 {Capillary Orlicz-Minkowski flow in the upper half-space}\thanks{\it {The research is partially supported by NSFC (No. 12261105).}}


\author{
Guanghan Li,~~Chenyang Liu*}
\address{
School of Mathematics and Statistics, Wuhan University, Wuhan 430072, China
}

\email{ghli@whu.edu.cn,liuchy@whu.edu.cn}

\thanks{$\ast$ Corresponding author}

\date{}
\maketitle
\begin{abstract}
In this paper, we study the long-time existence and asymptotic behavior of an anisotropic capillary Gauss curvature flow. By studying this flow and proving its convergence to a stationary solution,
we establish a new existence result for the capillary Orlicz-Minkowski problem without the evenness assumption, and provide a flow approach to the existence of smooth solutions.
\end{abstract}

\section{Introduction} \label{S1}

The classical Minkowski problem is a fundamental problem in the Brunn-Minkowski theory of convex geometry. Specifically, it asks whether there exists a closed convex hypersurface $\Sigma$ in $\mathbb{R}^{n+1}$
whose Gauss curvature is prescribed to be a given positive function $f$ on $\mathbb{S}^{n}$, that is
\begin{equation*}
K(\nu^{-1}(\xi))=f(\xi),\quad \forall\xi\in\mathbb{S}^{n},
\end{equation*}
where $\nu:\Sigma\rightarrow\mathbb{S}^{n}$ is the Gauss map of $\Sigma$ and $\nu^{-1}:\Sigma\rightarrow\mathbb{S}^{n}$ is the inverse of Gauss map. This problem was completely solved through the work of
Minkowski \cite{Min97,Min03}, Aleksandrov \cite{ADA56}, Lewy \cite{HL38}, Nirenberg \cite{LN53}, Pogorelov \cite{Pogo52,Pogo78}, Cheng-Yau \cite{C-Y76}, and many other mathematicians.
The Minkowski problem can be formulated analytically as the study of the Monge-Amp$\grave{\mathrm{e}}$re equation
\begin{equation*}
\det(\nabla^{2}h+h\sigma)=f^{-1},
\end{equation*}
where $h$ is the support function of convex set, $\sigma$ is the standard metric on $\mathbb{S}^{n}$, $\nabla^{2}$ is the Hessian operator and $f$ is a prescribed positive function. Hence, the study of such geometric problems is of significant
interest to the theory of fully nonlinear PDEs as well.

As a generalization of the classical Minkowski problem, the $L_p$-Minkowski problem was introduced by Lutwak \cite{Lut93} in the 1990s, building on Firey's concept of $p$-sums \cite{Fir62}.
The $L_p$-Minkowski problem holds an important position in convex geometry and geometric PDEs, which includes the logarithmic Minkowski problem and the centroaffine Minkowski problem.
For results on the $L_p$-Minkowski problem, one can refer to \cite{BBCY19,BBC20,BLYZ13,BIS19,CLZ17, CW06, HL13,JLW15,LYZ04,Z15}, and the references therein.

Building on the classical Minkowski problem and $L_p$-Minkowski problems, there has been significant progress in generalizing the related theory to capillary hypersurfaces, which we will detail in the following sections.

Let $\{E_{i}\}_{i=1}^{n+1}$ denote the standard orthonormal basis of $\mathbb{R}^{n+1}$, and define the upper Euclidean half-space as
\begin{equation*}
\mathbb{R}_{+}^{n+1}=\{X\in\mathbb{R}^{n+1}~|~\langle X,E_{n+1}\rangle>0\}.
\end{equation*}
We call a hypersurface $\Sigma$ in $\overline{\mathbb{R}_{+}^{n+1}}$ with boundary $\partial\Sigma\subset\partial\mathbb{R}_{+}^{n+1}$ capillary if it intersects with $\partial\Sigma\subset\partial\mathbb{R}_{+}^{n+1}$
at a constant contact angle $\theta\in(0,\pi)$. Let $\nu$ be the unit outword normal of $\Sigma$ in $\overline{\mathbb{R}_{+}^{n+1}}$, the contact angle $\theta$ is defined by
\begin{equation*}
\cos(\pi-\theta)=\langle \nu,e\rangle,
\end{equation*}
where $e:=-E_{n+1}$. If $\Sigma$ is a smooth convex capillary hypersurface in $\mathbb{R}_{+}^{n+1}$, then the domain $\widehat{\Sigma}$ bounded by $\Sigma$ and $\partial\mathbb{R}_{+}^{n+1}$ is called
a capillary convex body, and we denote by $\mathcal{K}_{\theta}$ the family of all capillary convex bodies in $\overline{\mathbb{R}_{+}^{n+1}}$. The capillary spherical cap $\mathcal{C}_{\theta}$ of unit radius is defined as
\begin{equation*}
\mathcal{C}_{\theta}=\{\xi\in\overline{\mathbb{R}_{+}^{n+1}}:|\xi-\cos\theta e|=1\}.
\end{equation*}
The capillary Gauss map $\tilde{\nu}:\Sigma\rightarrow\mathcal{C}_{\theta}$ is defined by
\begin{equation*}
\tilde{\nu}:=\nu(X)+\cos\theta e,
\end{equation*}
which is a diffeomorphism. Put $\ell(\xi)=\sin^{2}\theta+\cos\theta\langle \xi,e\rangle$, then one can check that $\ell$ gives the capillary support function of $\mathcal{C}_{\theta}$ (see \eqref{def-capillary ell} below).

Mei, Wang, and Weng \cite{MWW} pioneered the capillary Minkowski problem, defined as follows: given a positive smooth function $f$ on $\mathcal{C}_{\theta}$,
find a capillary convex body $\widehat{\Sigma} \in \mathcal{K}_{\theta}$ such that its Gauss curvature satisfies
\begin{equation*}
K(\tilde{\nu}^{-1}(\xi))=f(\xi),\quad \forall\xi\in\mathcal{C}_{\theta}.
\end{equation*}
Using the continuity method, they established the existence of smooth solutions to the capillary Minkowski problem for contact angles $\theta \in (0, \frac{\pi}{2}]$. By defining the capillary area measure of $\Sigma$ as
\begin{equation*}
m_{\theta}(\eta):=\int_{\eta} \frac{\sin^{2}\theta+\cos\theta\langle \xi,e\rangle}{K(\tilde{\nu}^{-1}(\xi))}d\xi
\end{equation*}
for all Borel sets $\eta\subset\mathcal{C}_{\theta}$, where $\tilde{\nu}^{-1}:\mathcal{C}_{\theta}\rightarrow\Sigma$ is the inverse capillary Gauss map and $d\xi$ is the area element of $\mathcal{C}_{\theta}$, the above capillary Minkowski problem is formulated to find a convex capillary hypersurface with contact angle $\theta$ such that its induced capillary area measure is equal to a given Borel measure on $C_{\theta}$.

Let $\xi=(\xi_{1},\cdot\cdot\cdot,\xi_{n},\xi_{n+1})$,
we say that a smooth function $f:\mathcal{C}_{\theta}\rightarrow\mathbb{R}$ is even if
\begin{equation*}
f(\xi_{1},\cdot\cdot\cdot,\xi_{n},\xi_{n+1})=f(-\xi_{1},\cdot\cdot\cdot,-\xi_{n},\xi_{n+1}),\quad\forall\xi\in\mathcal{C}_{\theta}.
\end{equation*}
A strictly convex capillary hypersurface is called even if its support function is even.

Research on the capillary $L_p$ Minkowski problem with $\theta \in (0, \frac{\pi}{2})$ has established existence results recently. Mei, Wang, and Weng first solved the case $p=1$ in \cite{MWW} (i.e. the capillary Minkowski problem).
They later extended this to the even case (i.e. when $f$ is even) for $1 < p < n+1$, and the non-even case $p \geq n+1$ in \cite{MWWlp}. Hu and Ivaki \cite{HI25} resolved the even case in the range $-n-1 < p < 1$ via an iterative scheme.
Subsequently, Hu, Hu and Ivaki \cite{HHI25} solve the capillary even $L_{p}$-Minkowski problem for all $p>-n-1$, and the non-even case for $p>n+1$ through geometric flows.

A further generalization of the $L_p$-Minkowski problem is the Orlicz-Minkowski problem, introduced by Haberl, Lutwak, Yang, and Zhang \cite{HLYZ10},
which in the smooth setting corresponds to solving a Monge-Ampère type equation on $\mathbb{S}^n$:
\begin{equation*}
\varphi(h)\det(\nabla^{2}h+h\sigma)=f^{-1},
\end{equation*}
where $h$, $\sigma$ and $\nabla^{2}$ have the same meaning as in the classical Minkowski problem, $\varphi$ and $f$ are some given smooth functions on $\mathbb{R}$ and $\mathbb{S}^n$ respectively.
When $\varphi(h)=h^{1-p}$, the Orlicz-Minkowski problem reduces to the $L_{p}$-Minkowski problem. This extension has motivated the development of the Orlicz-Brunn-Minkowski theory and the dual Orlicz-Brunn-Minkowski theory.
For further details, we refer to \cite{BIS21, GHW14, GHW15, JL19, LL20, LYZ10, ZZX14} and the references therein.

It is natural to consider the capillary Orlicz-Minkowski problem. In this paper, we mainly consider the existence of solutions to a more general class of capillary Orlicz-Minkowski problem, the corresponding equation is
\begin{equation}\label{omeq-1}
\left\{
\begin{aligned}
&\phi\left(\xi,\frac{h}{\ell}\right)\det(h_{ij}+h\delta_{ij})=f,\qquad&&~\mathrm{in}~\mathcal{C}_{\theta},\\
&\\
&\nabla_{\mu}h=\cot\theta h, \qquad&&~\mathrm{on}~ \partial \mathcal{C}_{\theta},\\
\end{aligned}
\right.
\end{equation}
where $\mu$ is the unit outward co-normal of $\partial \mathcal{C}_{\theta}$ in $\mathcal{C}_{\theta}$. In particular, taking $\phi\left(\xi,s\right)=s^{1-p}$, this recovers the capillary $L_{p}$-Minkowski problem.
Under appropriate assumptions on $\phi$, we establish the existence of smooth solutions to Eq. \eqref{omeq-1}. Namely, we obtain the following theorem:

\begin{theorem}\label{maintherorem1}
Let $\phi:\mathcal{C}_{\theta}\times(0,\infty)\rightarrow(0,\infty)$ be a $C^{\infty}$ function and $\theta\in(0,\frac{\pi}{2})$. For any given positive function $f\in C^{\infty}(\mathcal{C}_{\theta})$ satisfying
\begin{eqnarray} \label{condition1}
\limsup\limits_{s\rightarrow \infty}[\phi(\xi,s)s^{n}]<f<\liminf\limits_{s\rightarrow 0^{+}}[\phi(\xi,s)s^{n}],
\end{eqnarray}
there exists a smooth solution $h$ to Eq. \eqref{omeq-1}.
\end{theorem}

The same condition as \eqref{condition1} was used in \cite{BIS21, LL20} to solve the classical Orlicz-Minkowski problem and dual Orlicz-Minkowski problem, and therefore we believe that adopting it here is appropriate.
If $\phi\left(\xi,s\right)=s^{1-p}$ for $p>n+1$, then the condition \eqref{condition1} holds. Thus the solution to Eq. \eqref{omeq-1} can resolve the capillary $L_{p}$-Minkowski problem for $p>n+1$.
Inspired by \cite{BIS21, LL20}, we study the capillary Orlicz-Minkowski problem through an anisotropic capillary Gauss curvature flow.

Let $X_{0}:\mathcal{C}_{\theta}\rightarrow\overline{\mathbb{R}_{+}^{n+1}}$ be an embedding of a smooth, strictly convex capillary hypersurface $\Sigma_{0}=X_{0}(\mathcal{C}_{\theta})$ (see \eqref{def-capillary X} below). Consider a family of smooth,
strictly convex capillary hypersurfaces $\Sigma_{t}:=X(\mathcal{C}_{\theta}, t)$ satisfying
\begin{equation}\label{flow-1}
\left\{
\begin{aligned}
&\partial_{t}X=-fK\frac{\langle X,\nu\rangle}{\phi \ell}\tilde{\nu}+X, \qquad
&&~\mathrm{in}~\mathcal{C}_{\theta}\times[0,\infty),\\
&\tilde{\nu}\cdot e=0,  \qquad&&~\mathrm{on}~ \partial \mathcal{C}_{\theta}\times[0,\infty),\\
&X(\cdot,0)=X_{0}, \qquad &&~\mathrm{in}~\mathcal{C}_{\theta},
\end{aligned}
\right.
\end{equation}
where $\phi=\phi(X,\frac{\langle X,\nu\rangle}{\ell})$, $\ell=1+\cos\theta\langle \nu,e\rangle$.

For the classical case (without capillarity), similar flows have been used to study the classical Orlicz-Minkowski problem and dual Orlicz-Minkowski problem.
In the special case where $\phi\left(\xi,s\right)=s^{1-p}$ for $p>n+1$, flow \eqref{flow-1} reduces to the flow that was applied in \cite{HHI25} to resolve the capillary $L_{p}$-Minkowski problem.
The geometric flow generated by the Gauss curvature was first studied by Firey \cite{Fir74} to describe the shape of a tumbling stone. Since then, various Gauss curvature flows have been extensively studied and
have proven to be an efficient tool in the study of Minkowski-type problems, for example in \cite{DL23, DL25, GN, LWW, LL20} and references therein.
Recently, an isotropic capillary Gauss curvature flow was considered by Mei, Wang, and Weng \cite{MWWgf}, there have also been studies on anisotropic capillary Minkowski-type problems and geometric inequalities, one can refer to \cite{DGL, GL}.

Theorem \ref{maintherorem1} is a consequence of the following long-time existence and convergence of  flow \eqref{flow-1}.

\begin{theorem}\label{maintherorem2}
Let $\theta\in(0,\frac{\pi}{2})$, $\Sigma_{0}$ be a smooth, strictly convex capillary hypersurface in $\overline{\mathbb{R}_{+}^{n+1}}$ with positive capillary support function. Assume $f\in C^{\infty}(\mathcal{C}_{\theta})$ and
$\phi\in C^{\infty}(\mathcal{C}_{\theta}\times(0,\infty))$ are positive functions satisfying the condition \eqref{condition1}. Then flow \eqref{flow-1} has a smooth, strictly convex solution $\Sigma_{t}$ for all time $t>0$,
and a subsequence of $\Sigma_{t}$ converges in $C^{\infty}$-topology to a smooth, strictly convex capillary hypersurface $\Sigma$ satisfying Eq. \eqref{omeq-1}.
\end{theorem}

Using methods from elliptic equations, Wang and Zhu \cite{WZ25} also rasied the capillary Orlicz-Minkowski problem. They prove that if $\tilde{\phi}:[0,+\infty)\rightarrow[0,+\infty)$ is a $C^{2}$-smooth convex function,
then the capillary Orlicz-Minkowski problem is equivalent to solve the following Robin boundary value problem of the Monge-Amp$\grave{\mathrm{e}}$re type equation:
\begin{equation}\label{WZ-omeq-1}
\left\{
\begin{aligned}
&\tilde{\phi}\left(\frac{\ell}{h}\right)h\det(h_{ij}+h\delta_{ij})=f,\qquad&&~\mathrm{in}~\mathcal{C}_{\theta},\\
&\\
&\nabla_{\mu}h=\cot\theta h, \qquad&&~\mathrm{on}~ \partial \mathcal{C}_{\theta}.\\
\end{aligned}
\right.
\end{equation}
 They proved the following result.

\begin{theoremA}\cite[Theorem 1.1]{WZ25}
Let $\tilde{\phi}:[0,+\infty)\rightarrow[0,+\infty)$ be a $C^{2}$-smooth, strictly increasing, convex, and log-concave function satisfying $\tilde{\phi}(0)=0$, $\lim\limits_{t\rightarrow 0^{+}}\tilde{\phi}'(t)=0$,
$\frac{d}{dt}\log\frac{\tilde{\phi}(t)}{t}\geq 0$ for all $t>0$.
For $\theta \in (0, \frac{\pi}{2})$, let $f\in C^{2}(\mathcal{C}_{\theta})$ be an even, positive function and satisfy
\begin{equation*}
\frac{1}{n+1}\int_{\mathcal{C}_{\theta}}f\geq\tilde{\phi}(|\widehat{\mathcal{C}_{\theta}}^{\frac{1}{n+1}}|).
\end{equation*}
Then there exists a smooth, symmetric $\widehat{\Sigma}\in\mathcal{K}^{o}_{\theta}$ (i.e. the family of all capillary convex bodies enclosing the origin) with $|\widehat{\Sigma}|=1$,
such that its support function $h$ solves \eqref{WZ-omeq-1} and satisfies the orthogonality condition with respect to $f$.
\end{theoremA}

In Eq. \eqref{omeq-1}, the case $\phi\left(\xi,s\right)=\tilde{\phi}\left(1/s\right)s$ corresponds to the capillary Orlicz-Minkowski problem (Eq. \eqref{WZ-omeq-1}).
Theorem \ref{maintherorem1} solves the capillary Orlicz-Minkowski problems in non-even case, note that the condition \eqref{condition1} and the requirement on $f$, $\tilde{\phi}$ in Theorem A are independent. For further details, one can refer to \cite{WZ25}.

This paper is organized as follows. In section 2, we give some basic knowledge about capillary hypersurfaces and flow \eqref{flow-1} in the Euclidean half-space. In section 3, the long-time existence of flow \eqref{flow-1} will be proved.
The key step is to establish uniform positive upper and lower bounds for capillary support function and the principal curvatures along the flow. To circumvent the extra boundary difficulties in obtaining
$C^{2}$ estimates along the flow, we first establish lower and upper bounds on the Gauss curvature and then get the upper bound of the principal radii of curvature. In section 4, by considering a monotone decreasing geometric functional along the flow,
we prove that the smooth solution of flow \eqref{flow-1} can converge for a subsequence to a smooth solution of Eq. \eqref{omeq-1}, and thus complete the proof of Theorems \ref{maintherorem1} and \ref{maintherorem2}.

\section{Preliminaries} \label{S2}

In this section, we provide some basic knowledge about capillary hypersurfaces in $\overline{\mathbb{R}_{+}^{n+1}}$. For more details, one can refer to \cite{HHI25,MWW,MWWgf,MWWXAF}.

Let $M$ be a compact orientable smooth manifold of
dimension $n$ with boundary $\partial M$ and let $X:M\rightarrow\overline{\mathbb{R}_{+}^{n+1}}$ be a properly embedded smooth hypersurface satisfying
\begin{equation*}
X(\mathrm{int}~M)\subset\mathbb{R}_{+}^{n+1}\quad\mathrm{and}\quad X(\partial M)\subset\partial\mathbb{R}_{+}^{n+1}.
\end{equation*}
We denote that $\Sigma=X(M)$ and $\partial\Sigma=X(\partial M)$. Then denote by $\widehat{\partial\Sigma}$ the region enclosed by $\partial\Sigma$ in $\partial\mathbb{R}_{+}^{n+1}$, and by $\widehat{\Sigma}$ the region bounded by $\Sigma$ and $\partial\mathbb{R}_{+}^{n+1}$. Assuming $\Sigma$ is a strictly convex hypersurface, meaning that $\Sigma$ together with $\partial\mathbb{R}_{+}^{n+1}$ bounds a convex body and that $\Sigma$ has a positive definite second fundamental form. Consequently, $\partial\Sigma$ is a closed, strcitly
convex hypersurfce in $\partial\mathbb{R}_{+}^{n+1}$ (see, e.g., \cite[Corollary 2.5]{WWX24}).

A strictly convex hypersurface $\Sigma$ that intersects $\partial\Sigma\subset\partial\mathbb{R}_{+}^{n+1}$ at a constant contact angle $\theta \in (0, \pi)$ is called a capillary strictly convex hypersurface.
For any $r>0$, a simple example of a capillary strictly convex hypersurface is the capillary spherical cap
\begin{equation*}
\mathcal{C}_{\theta,r}=\{\xi\in\overline{\mathbb{R}_{+}^{n+1}}:|\xi-r\cos\theta e|=r\},
\end{equation*}
where $e:=-E_{n+1}$ is the unit outward normal of $\partial\mathbb{R}_{+}^{n+1}\subset\overline{\mathbb{R}_{+}^{n+1}}$. For simplicity, let $\mathcal{C}_{\theta}:=\mathcal{C}_{\theta,1}$ be the unit spherical cap.

Take $M=\mathcal{C}_{\theta}$ and consider a strictly convex capillary hypersurface $\Sigma=X(\mathcal{C}_{\theta})$. Let $\nu$ be the unit outward normal of $\Sigma$ in $\overline{\mathbb{R}_{+}^{n+1}}$ and
$\mu$ be the unit outward co-normal of $\partial\Sigma$ in $\Sigma$. Denote by $D$ the Euclidean covariant derivative, denote by $\sigma$ and $\nabla$ the spherical metric and the covariant derivative on $\mathcal{C}_{\theta}$, respectively.
Following \cite{MWWXAF}, the Gauss map $\nu:\Sigma\rightarrow\mathbb{S}^{n}$ has its image in
\begin{equation*}
\mathbb{S}_{\theta}^{n}:=\{y\in\mathbb{S}^{n}~|~y_{n+1}\geq\cos\theta\}.
\end{equation*}
Let $T:\mathbb{S}_{\theta}^{n}\rightarrow\mathcal{C}_{\theta}$ be the translation given by
\begin{equation*}
T(y)=y+\cos\theta e.
\end{equation*}
Therefore, the capillary Gauss map $\tilde{\nu}:\Sigma\rightarrow\mathcal{C}_{\theta}$ is defined as
\begin{equation*}
\tilde{\nu}:=T\circ \nu=\nu+\cos\theta e.
\end{equation*}
From \cite[Lemma 2.2]{MWWXAF}, we can parametrize $\Sigma$ by the inverse capillary Gauss map, i.e., $X:\mathcal{C}_{\theta}\rightarrow \Sigma$, given by
\begin{equation}\label{def-capillary X}
X(\xi)=\tilde{\nu}^{-1}(\xi)=\nu^{-1}\circ T^{-1}(\xi)=\nu^{-1}(\xi-\cos\theta e),
\end{equation}
where $\xi\in\mathcal{C}_{\theta}$. The support function of $\Sigma$ is given by
\begin{equation*}
h(X)=\langle X,\nu(X)\rangle.
\end{equation*}
By the inverse capillary Gauss map, $h$ is regarded as a function on $\mathcal{C}_{\theta}$,
\begin{equation}\label{def-capillary sup}
h(\xi)=\langle X(\xi),\nu(X(\xi))\rangle=\langle X(\xi),T^{-1}(\xi)\rangle=\langle \tilde{\nu}^{-1}(\xi),\xi-\cos\theta e\rangle.
\end{equation}
The function $h$ in the parametrization \eqref{def-capillary sup} is called \emph{the capillary support function} of $\Sigma$ and $h$ satisfies the capillary boundary condition \cite[Lemma 2.4]{MWWXAF}
\begin{equation*}
\nabla_{\mu}h=\cot\theta h,\quad\mathrm{on}~\partial\mathcal{C}_{\theta}.
\end{equation*}
It is clear that the capillary Gauss map for $\mathcal{C}_{\theta}$ is the identity map from $\mathcal{C}_{\theta}\rightarrow\mathcal{C}_{\theta}$ and
\begin{equation*}
h_{\mathcal{C}_{\theta}}(\xi)=\langle \xi,\xi-\cos\theta e\rangle=\sin^{2}\theta+\cos\theta\langle \xi,e\rangle=1+\cos\theta\langle \nu,e\rangle.
\end{equation*}
For simplicity, we denote
\begin{equation}\label{def-capillary ell}
\ell(\xi):=\sin^{2}\theta+\cos\theta\langle \xi,e\rangle=1+\cos\theta\langle \nu,e\rangle,
\end{equation}
and $\ell$ also satisfies the capillary boundary condition
\begin{equation*}
\nabla_{\mu}\ell=\cot\theta \ell,\quad\mathrm{on}~\partial\mathcal{C}_{\theta}.
\end{equation*}
Therefore $\ell$ is the capillary support function of $\mathcal{C}_{\theta}$. We state two useful propositions below, their detailed proofs can be found in \cite{MWW,MWWXAF}, here we omit the proofs and only list the statements.

\begin{proposition}\label{prop-1}
\cite[Lemma 2.4]{MWWXAF} Choosing a local orthonormal frame $\{e_{i}\}_{i=1}^{n}$ on $\mathcal{C}_{\theta}$. The second fundamental form $b_{ij}$ of $\Sigma$ is given by
\begin{equation*}
b_{ij}=\nabla^{2}_{ij}h+h\delta_{ij},
\end{equation*}
where $\nabla^{2}_{ij}h=\nabla^{2}h(e_{i}, e_{j})$, and the Gauss curvature of $\Sigma$ at $X(\xi)$ is
\begin{equation*}
K(X(\xi))=\frac{1}{\det(\nabla^{2}h(\xi)+h(\xi)I)}.
\end{equation*}
\end{proposition}

Along the boundary $\partial\mathcal{C}_{\theta}$, we choose an orthonormal frame $\{e_{i}\}_{i=1}^{n}$ with $e_{n}=\mu$, where $\mu$ is again the unit co-normal of $\partial\mathcal{C}_{\theta}\subset\mathcal{C}_{\theta}$.

\begin{proposition}\label{prop-2}
\cite[Proposition 2.8]{MWW} The capillary support function $h$ satisfies the following boundary conditions on $\partial\mathcal{C}_{\theta}$:
\begin{equation}\label{prop-2-eq1}
\nabla^{2}_{\alpha n}h=0,
\end{equation}
and
\begin{equation}\label{prop-2-eq2}
\nabla^{2}_{\alpha n}\left(\frac{h}{\ell}\right)=-\cot\theta\nabla_{\alpha}\left(\frac{h}{\ell}\right),
\end{equation}
where $\alpha=1, 2, \cdots, n-1$.
\end{proposition}

For convenience, when dealing with the tensors and their covariant derivatives on the spherical cap $\mathcal{C}_{\theta}$, we will use a local frame to express tensors with the help of their
components and indices appearing after the semi-colon denote the covariant derivatives. Choosing a local orthonormal frame $\{e_{i}\}_{i=1}^{n}$ on $\mathcal{C}_{\theta}$, we use the notations $h_{i}=\nabla h(e_{i})$,
$h_{ij}=\nabla^{2}h(e_{i}, e_{j})$ for the covariant derivatives of $h$, and $W_{ij;k}:=\nabla_{e_{k}}(W_{ij})$ etc. For the standard metric on spherical cap $\mathcal{C}_{\theta}$, we have the commutator formulate
\begin{equation}\label{ricci-eq1}
h_{kij}=h_{ijk}+h_{k}\delta_{ij}-h_{j}\delta_{ki},
\end{equation}
and
\begin{equation}\label{ricci-eq2}
h_{klij}=h_{ijkl}+2h_{kl}\delta_{ij}-2h_{ij}\delta_{kl}+h_{li}\delta_{kj}-h_{kj}\delta_{il}.
\end{equation}

Using the capillary Gauss map parametrization, flow \eqref{flow-1} is transformed into a parabolic equation subject to a Robin boundary condition as follows (cf. \cite[Proposition 3.2]{MWWgf}).

\begin{proposition}\label{prop-3}
Let $\theta\in(0,\frac{\pi}{2})$ and $\Sigma_{t}$ be the solution to \eqref{flow-1}, then the capillary support function $h(\xi,t)$ of $\Sigma_{t}$ satisfies the following equations:
\begin{equation}\label{flow-2}
\left\{
\begin{aligned}
&\partial_{t} h=-fhK\frac{1}{\phi(\xi,\frac{h}{\ell})}+h, \qquad
&&~\mathrm{in}~
\mathcal{C}_{\theta}\times[0,\infty),\\
&\nabla_{\mu}h=\cot\theta h, \qquad&&~\mathrm{on}~ \partial \mathcal{C}_{\theta}\times[0,\infty),\\
&h(\xi,0)=h_{0}(\xi), \qquad &&~\mathrm{in}~\mathcal{C}_{\theta},
\end{aligned}
\right.
\end{equation}
where $h_{0}$ is the capillary support function of $\Sigma_{0}$.
\end{proposition}

\section{Long-time existence of the flow}  \label{S3}

In this section, we will give a priori estimates about capillary support function and obtain the long-time existence of the flow \eqref{flow-1}.

We begin with the estimate of the upper and the lower bounds of $h$.

\begin{lemma}\label{lemmaC0}
Let $\theta\in(0,\frac{\pi}{2})$. Suppose $h$ is a smooth, strictly convex solution to flow \eqref{flow-2} on $\mathcal{C}_{\theta}\times[0,T)$, and $f$, $\phi$ are smooth, positive functions satisfying \eqref{condition1}. Then
\begin{equation*}
\frac{1}{C}\leq h(\xi,t) \leq C,
\end{equation*}
where $C$ is a positive constant depending on $n$, $\theta$, $\min_{\mathcal{C}_{\theta}}f$, $\max_{\mathcal{C}_{\theta}}f$, $\Sigma_{0}$, $\limsup\limits_{s\rightarrow \infty}[\phi(\xi,s)s^{n}]$ and
$\liminf\limits_{s\rightarrow 0^{+}}[\phi(\xi,s)s^{n}]$.
\end{lemma}
\begin{proof}
From the evolution equation \eqref{flow-2}, a direct computation gives
\begin{equation}\label{C0-eq1}
\partial_{t}\left(\frac{h}{\ell}\right)=\frac{h}{\ell}\left(-f\sigma_{n}^{-1}\frac{1}{\phi(\xi,\frac{h}{\ell})}+1 \right),
\end{equation}
and
\begin{equation}\label{C0-bd1}
\nabla_{\mu}\left(\frac{h}{\ell}\right)=0,\quad\mathrm{on}~\partial \mathcal{C}_{\theta},
\end{equation}
where we have used the fact that $\nabla_{\mu}\ell=\cot\theta\ell$ on $\partial \mathcal{C}_{\theta}$.

We claim that for any fixed $t>0$, if the function $\frac{h}{\ell}(\cdot,t)$ attains its minimum at a point $\xi_{0}\in\mathcal{C}_{\theta}$, then at this point,
\begin{equation}\label{C0-claim1}
\nabla\left(\frac{h}{\ell}\right)=0, \quad \nabla^{2}\left(\frac{h}{\ell}\right)\geq0.
\end{equation}
If $\xi_{0}\in\mathcal{C}_{\theta} \backslash \partial\mathcal{C}_{\theta}$, then the claim is obvious. If $\xi_{0}\in \partial\mathcal{C}_{\theta}$, let $\{e_{i}\}_{i=1}^{n}$ be the orthonormal frame around $\xi_{0}$
such that $e_{n}=\mu$. For any $e_{\alpha}\in T_{\xi_{0}}\partial \mathcal{C}_{\theta}$, we have
\begin{equation*}
\nabla_{\alpha}\left(\frac{h}{\ell}(\xi_{0},t)\right)=0\quad\mathrm{for}\quad 1\leq \alpha\leq n-1.
\end{equation*}
Then the boundary value condition \eqref{C0-bd1} implies that $\nabla\left(\frac{h}{\ell}\right)=0$ at $\xi_{0}$. The minimality of $\xi_{0}$, implies that
\begin{equation*}
\nabla^{2}_{\alpha\beta}\left(\frac{h}{\ell}(\xi_{0},t)\right)\geq0\quad\mathrm{for}\quad 1\leq \alpha, \beta \leq n-1.
\end{equation*}
From \eqref{prop-2-eq2} and the condition $\nabla \left(\frac{h}{\ell}\right)=0$, we get
\begin{equation*}
\nabla^{2}\left(\frac{h}{\ell}(\xi_{0},t)\right)(e_{\alpha},e_{n})=-\cot\theta\nabla_{\alpha} \left(\frac{h}{\ell}(\xi_{0},t)\right)=0\quad\mathrm{for}\quad 1\leq \alpha \leq n-1.
\end{equation*}
Let $\gamma(s)$ be the geodesic in $\mathcal{C}_{\theta}$ starting from $\gamma(0)=\xi_{0}$ with $\gamma'(0)=-e_{n}$ for $s\in[0,\epsilon]$ and $\epsilon(\epsilon<\theta)$ sufficiently small,
then the minimality of $\xi_{0}$ implies
\begin{equation*}
0\leq\frac{d^{2}}{ds^{2}}\bigg |_{s=0}\left(\frac{h}{\ell}(\gamma(s),t)\right)=\nabla^{2}\left(\frac{h}{\ell}(\gamma(0),t)\right)(\gamma'(0),\gamma'(0))+\langle\nabla\left(\frac{h}{\ell}(\gamma(0),t)\right),\gamma''(0)\rangle.
\end{equation*}
Together with the condition $\nabla \left(\frac{h}{\ell}\right)=0$, we conclude that
\begin{equation*}
\nabla^{2}\left(\frac{h}{\ell}(\xi_{0},t)\right)(e_{n},e_{n})\geq 0.
\end{equation*}
Consequently, \eqref{C0-claim1} also hold in the case $\xi_{0}\in\partial\mathcal{C}_{\theta}$.

At the point $\xi_{0}$,
\begin{equation*}
\begin{split}
0\leq \nabla^{2}_{ij}\frac{h}{\ell}
&=\frac{\nabla^{2}_{ij}h}{\ell}+\frac{2h}{\ell^{3}}\nabla_{i}\ell\nabla_{j}\ell-\frac{\nabla^{2}_{ij}\ell}{\ell^{2}}h-\frac{\nabla_{j}h\nabla_{i}\ell}{\ell^{2}}-\frac{\nabla_{i}h\nabla_{j}\ell}{\ell^{2}}\\
&=\frac{b_{ij}}{\ell}-\frac{1}{\ell}\left(\nabla_{i}\frac{h}{\ell}\nabla_{j}\ell+\nabla_{j}\frac{h}{\ell}\nabla_{i}\ell\right)-\frac{h}{\ell^{2}}\delta_{ij}\\
&=\frac{b_{ij}}{\ell}-\frac{h}{\ell^{2}}\delta_{ij},
\end{split}
\end{equation*}
where we used $\nabla \left(\frac{h}{\ell}\right)=0$ at $\xi_{0}$ and $\nabla^{2}_{ij}\ell+\ell \delta_{ij}=\delta_{ij}$. Multiplying both sides of the above inequality by $\ell>0$ yields
\begin{equation*}
b_{ij}\geq\frac{h}{\ell}\delta_{ij}.
\end{equation*}
Since $\sigma_{n}$ is monotonically increasing in its argument, it follows that
\begin{equation}\label{C0-ineq1}
\sigma_{n}(b_{ij})\geq\sigma_{n}\left(\frac{h}{\ell}\delta_{ij}\right)=\left(\frac{h}{\ell}\right)^{n}.
\end{equation}
Put $u(t)=\min_{\mathcal{C}_{\theta}}(\frac{h}{\ell})(\cdot,t)=(\frac{h}{\ell})(\xi_{0},t)$, substituting \eqref{C0-ineq1} into \eqref{C0-eq1}, we find
\begin{equation*}
\frac{\partial u}{\partial t}\geq u(-fu^{-n}\frac{1}{\phi(\xi_{0},u)}+1).
\end{equation*}
Let $\underline{A}=\liminf\limits_{s\rightarrow 0^{+}}[\phi(\xi,s)s^{n}]$. By condition \eqref{condition1}, $\varepsilon=\frac{1}{2}(\underline{A}-\max_{\mathcal{C}_{\theta}}f(\xi))$ is positive. Then, there exists a positive constant
$C_{1}$ depending on $n$, $\theta$, $\max_{\mathcal{C}_{\theta}}f(\xi)$ and $\liminf\limits_{s\rightarrow 0^{+}}[\phi(\xi,s)s^{n}]$, such that
\begin{equation*}
u^{n}\phi(\xi_{0},u)>\underline{A}-\varepsilon
\end{equation*}
for $u<C_{1}$. It follows by \eqref{condition1}
\begin{equation*}
u^{n}\phi(\xi_{0},u)-f(\xi)>\underline{A}-\varepsilon-\max_{\mathcal{C}_{\theta}}f(\xi)>0,
\end{equation*}
which shows that
\begin{equation*}
\frac{\partial u}{\partial t}>0
\end{equation*}
at minimal points. Hence
\begin{equation*}
u\geq \min\{C_{1},\min_{\mathcal{C}_{\theta}}u(\xi,0)\}.
\end{equation*}
Similarly, we can show that
\begin{equation*}
u\leq \max\{C_{2},\max_{\mathcal{C}_{\theta}}u(\xi,0)\},
\end{equation*}
where $C_{2}$ is a positive constant depending on $n$, $\theta$, $\min_{\mathcal{C}_{\theta}}f(\xi)$ and $\limsup\limits_{s\rightarrow \infty}[\phi(\xi,s)s^{n}]$.

Recall that $\ell$ is bounded from above and below by a positive constant $C$ depending only on $\theta$. Thus, the bounds on $u$ give the bounds on $h$, this completes the proof.
\end{proof}

\begin{lemma}\label{lemmaC1}
Let $\theta\in(0,\frac{\pi}{2})$. Suppose $h$ is a smooth, strictly convex solution to flow \eqref{flow-2} on $\mathcal{C}_{\theta}\times[0,T)$, and $f$, $\phi$ are smooth, positive functions satisfying \eqref{condition1}, then
\begin{equation}\label{lemmaC1-ieq1}
|\nabla h(\xi,t)| \leq C, \quad \forall (\xi,t)\in \mathcal{C}_{\theta}\times[0,T),
\end{equation}
where $C$ is a positive constant depending on $n$, $\theta$, $\min_{\mathcal{C}_{\theta}}f$, $\max_{\mathcal{C}_{\theta}}f$, $\Sigma_{0}$, $\limsup\limits_{s\rightarrow \infty}[\phi(\xi,s)s^{n}]$ and
$\liminf\limits_{s\rightarrow 0^{+}}[\phi(\xi,s)s^{n}]$.
\end{lemma}
\begin{proof}
Define the function
\begin{equation*}
P:=|\nabla h|^{2}+h^{2}.
\end{equation*}
For any fixed $t\in[0,T)$, suppose that $P$ attains its maximum value at some point, say $\xi_{0}\in\mathcal{C}_{\theta}$.

\textbf{Case $1$}. If $\xi_{0}\in\mathcal{C}_{\theta} \backslash \partial\mathcal{C}_{\theta}$, by the maximal condition,
\begin{equation*}
0=\nabla_{e_{i}}P=2\sum\limits_{k=1}^{n}h_{k}h_{ki}+2hh_{i},\quad \mathrm{for}~1\leq i\leq n,
\end{equation*}
together with the fact that $\nabla^{2}h+h\sigma>0$, we obtain $\nabla h(\xi_{0},t)=0$. According to Lemma \ref{lemmaC0}, \eqref{lemmaC1-ieq1} holds in the case.

\textbf{Case $2$}. If $\xi_{0}\in\partial\mathcal{C}_{\theta}$, we choose an orthonormal frame $\{e_{\alpha}\}_{\alpha=1}^{n-1}$ of $\mathcal{C}_{\theta}$ such that $\{(e_{\alpha})_{\alpha=1}^{n-1},e_{n}=\mu\}$ forms an
orthonormal frame of $\mathcal{C}_{\theta}$. From Propostion \ref{prop-2}, we know that $h_{n\alpha}(\xi_{0},t)=0$. Using the maximal condition again, we have
\begin{equation*}
0=\nabla_{e_{\alpha}}P=2\sum\limits_{i=1}^{n}h_{i}h_{i\alpha}+2hh_{\alpha},
\end{equation*}
which implies
\begin{equation*}
(\nabla_{e_{\alpha}}h)(\xi_{0},t)=0.
\end{equation*}
Combining this with the boundary condition $\nabla_{\mu}h=\cot\theta h$, we obtain
\begin{equation*}
P\leq h_{n}^{2}(\xi_{0},t)+h^{2}(\xi_{0},t)=(1+\cot^{2}\theta)h^{2}(\xi_{0},t).
\end{equation*}
According to Lemma \ref{lemmaC0}, we know that \eqref{lemmaC1-ieq1} holds and the proof is complete.
\end{proof}

Next we will establish the uniform upper and lower bounds for the principal curvatures of the flow \eqref{flow-2}. We first derive lower and upper bounds on the Gauss curvature and then obtain the upper bound of the principal radii of curvature.

By Lemmas \ref{lemmaC0} and \ref{lemmaC1}, if $h$ is a smooth, strictly convex solution of \eqref{flow-2} on $\mathcal{C}_{\theta}\times[0,T)$, and $f$, $\phi$ are smooth, positive functions satisfying \eqref{condition1}, then along the flow for $[0,T)$,
$||h||_{C^{0}(\mathcal{C}_{\theta}\times[0,T))}$ and $||h||_{C^{1}(\mathcal{C}_{\theta}\times[0,T))}$ are bounded by positive constants independent of $t$. Moreover, since $\ell$ is bounded by a constant depending only on $\theta$, $h/\ell$ is a smooth function whose range is contained in a bounded interval $I_{[0,T)}$. Here $I_{[0,T)}$ depends only on $\theta$, the upper and lower bounds of $h$ on $[0,T)$, and is independent of $t$.

\begin{lemma}\label{lemmaC2-1}
Let $\theta\in(0,\frac{\pi}{2})$. Suppose $h$ is a smooth, strictly convex solution to flow \eqref{flow-2} on $\mathcal{C}_{\theta}\times[0,T)$, and $f$, $\phi$ are smooth, positive functions satisfying \eqref{condition1}, then
\begin{equation*}
K(\xi,t)\leq C,
\end{equation*}
where $C$ is a positive constant depending on $n$, $\theta$, $\min_{\mathcal{C}_{\theta}}f$, $\max_{\mathcal{C}_{\theta}}f$, $||h||_{C^{0}(\mathcal{C}_{\theta}\times[0,T))}$,
$||\phi||_{C^{0}(\mathcal{C}_{\theta}\times I_{[0,T)})}$, $||\phi||_{C^{1}(\mathcal{C}_{\theta}\times I_{[0,T)})}$ and $\Sigma_{0}$.
\end{lemma}

\begin{proof}
For convenience, we denote $\phi(\xi,\frac{h}{\ell})^{-1}=\psi(\xi,\frac{h}{\ell})$ below (or simply $\psi$). Consider the following auxiliary function
\begin{equation}\label{lemmaC2-1eq1}
Q=\frac{fhK\psi-h}{h-\varepsilon_{0}}=\frac{-h_{t}}{h-\epsilon_{0}},
\end{equation}
where $h_{t}:=\partial_{t}h$ and
\begin{equation*}
\varepsilon_{0}:=\frac{1}{2}\min\limits_{\mathcal{C}_{\theta}\times[0,T)} h(\xi,t)>0.
\end{equation*}

Assume that $\max K\gg 1$, otherwise, the proof is complete. By Lemma \ref{lemmaC0}, the upper bound of $K(\xi,t)$ follows from that of $Q(\xi,t)$. Therefore, it suffices to bound $Q(\xi,t)$ from above.
In what follows, we also assume that $\max Q\gg 1$.

For any fixed $t\in[0,T)$, suppose the maximum of $Q(\xi,t)$ is attained at some point $\xi_{0}\in\mathcal{C}_{\theta}$. On $\partial\mathcal{C}_{\theta}$, we have
\begin{equation*}
\nabla_{\mu}Q=-\frac{h_{t\mu}}{h-\varepsilon_{0}}+\frac{h_{t}h_{\mu}}{(h-\varepsilon_{0})^{2}}=\frac{\varepsilon_{0}\cot\theta h_{t}}{(h-\varepsilon_{0})^{2}}
=-\frac{\varepsilon_{0}\cot\theta Q}{h-\varepsilon_{0}}<0.
\end{equation*}
Consequently, the maximum of $Q(\cdot,t)$ is attained at some point $\xi_{0}\in\mathcal{C}_{\theta} \backslash \partial\mathcal{C}_{\theta}$.

Choose a local orthonormal frame $\{e_{i}\}_{i=1}^{n}$ such that $\{b_{ij}\}$ is diagonal at the point $\xi_{0}$. Then, at the point $\xi_{0}$, we have
\begin{equation}\label{lemmaC2-1eq2}
0=\nabla_{i}Q=\frac{-h_{ti}}{h-\varepsilon_{0}}+\frac{h_{t}h_{i}}{(h-\varepsilon_{0})^{2}},
\end{equation}

\begin{equation}\label{lemmaC2-1ieq3}
\begin{split}
0\geq\nabla^{2}_{ii}Q
=&\frac{-h_{tii}}{h-\varepsilon_{0}}+\frac{h_{t}h_{ii}}{(h-\varepsilon_{0})^{2}}+\frac{2h_{ti}h_{i}}{(h-\varepsilon_{0})^{2}}-\frac{2h_{t}h_{i}^{2}}{(h-\varepsilon_{0})^{3}}\\
=&\frac{-h_{tii}}{h-\varepsilon_{0}}+\frac{h_{t}h_{ii}}{(h-\varepsilon_{0})^{2}}.
\end{split}
\end{equation}
In the last equality we used \eqref{lemmaC2-1eq2}. This implies that
\begin{equation}\label{lemmaC2-1ieq4}
-h_{tii}-h_{t}\leq-\frac{h_{t}h_{ii}}{h-\varepsilon_{0}}-h_{t}=-\frac{h_{t}}{h-\varepsilon_{0}}\left(h_{ii}+(h-\varepsilon_{0})\right)=Q(b_{ii}-\varepsilon_{0}).
\end{equation}
Moreover, we have
\begin{equation}\label{lemmaC2-1eq5}
\begin{split}
\partial_{t}Q
=&~\frac{-h_{tt}}{h-\varepsilon_{0}}+\frac{h_{t}^{2}}{(h-\varepsilon_{0})^{2}}\\
=&~\frac{fhK\psi}{h-\varepsilon_{0}}\left(\frac{h_{t}}{h}+\frac{K_{t}}{K}+\frac{\psi_{t}}{\psi}\right)+Q+Q^{2}.
\end{split}
\end{equation}
Denote
\begin{equation}\label{def-u}
u=\frac{h}{\ell},\quad\psi(\xi,u)=\psi(\xi,\frac{h}{\ell}).
\end{equation}
According to the definition of $u$ and $\psi$, it follows that
\begin{equation}\label{lemmaC2-1eq6}
\frac{\psi_{t}}{\psi}=\frac{\partial_{u}\psi}{\psi}u_{t}=-\frac{\partial_{u}\phi}{\phi}\cdot\frac{h_{t}}{\ell},
\end{equation}
where $\partial_{u}\psi(\xi,u)$ denote the partial derivative of $\psi$ with respect to its second variable $u$, and $\partial_{u}\phi$ denotes the partial derivative of $\phi$ with respect to $u$.
For simplicity, let us put
\begin{equation*}
\sigma_{n}=\det\left(\nabla^{2}h+h\emph{I}\right),\quad \sigma_{n}^{ij}=\frac{\partial\det\left(\nabla^{2}h+h\emph{I}\right)}{\partial b_{ij}}.
\end{equation*}
Using \eqref{lemmaC2-1ieq4}, at the point $\xi_{0}$, we calculate that
\begin{equation}\label{lemmaC2-1ieq7}
\begin{split}
K_{t}=\partial_{t}(\sigma_{n}^{-1})&=-\sigma_{n}^{-2}\sum\limits_{i,j}\sigma_{n}^{ij}(h_{tij}+h_{t}\delta_{ij})\\
&\leq\sigma_{n}^{-2}\sum\limits_{i}\sigma_{n}^{ii}Q(b_{ii}-\varepsilon_{0})\\
&=KQ(n-\varepsilon_{0}\sum\limits_{i}b^{ii}),
\end{split}
\end{equation}
where $\{b^{ij}\}$ denotes the inverse matrix of $\{b_{ij}\}$. Here, we used $\sigma_{n}^{ii}=\sigma_{n}b^{ii}$ and
\begin{equation*}
\sum\limits_{i}\sigma_{n}^{ii}b_{ii}=n\sigma_{n},\quad \sum\limits_{i}\sigma_{n}^{ii}=\sigma_{n}\sum\limits_{i}b^{ii}.
\end{equation*}
It follows from \eqref{lemmaC2-1ieq7} and $\sum_{i}b^{ii}\geq n(\Pi_{i}b^{ii})^{\frac{1}{n}}=nK^{\frac{1}{n}}$ that
\begin{equation}\label{lemmaC2-1ieq8}
\frac{K_{t}}{K}\leq nQ(1-\varepsilon_{0}K^{\frac{1}{n}}).
\end{equation}
Substituting \eqref{lemmaC2-1eq1}, \eqref{lemmaC2-1eq6} and \eqref{lemmaC2-1ieq8} into \eqref{lemmaC2-1eq5}, we obtain
\begin{equation*}
\begin{split}
\partial_{t}Q
&\leq\frac{fhK\psi}{h-\varepsilon_{0}}\left(\frac{h_{t}}{h}+nQ(1-\varepsilon_{0}K^{\frac{1}{n}})-\frac{\partial_{u}\phi}{\phi\ell}h_{t}\right)+Q+Q^{2}\\
&=(Q+\frac{h}{h-\varepsilon_{0}})\left(-\frac{Q(h-\varepsilon_{0})}{h}+nQ(1-\varepsilon_{0}K^{\frac{1}{n}})+\frac{\partial_{u}\phi}{\phi\ell}Q(h-\varepsilon_{0})\right)+Q+Q^{2}\\
&=Q^{2}\left(\frac{h-\varepsilon_{0}}{h}(\frac{\partial_{u}\phi}{\phi\ell}h-1)+(n+1-n\varepsilon_{0}K^{\frac{1}{n}})\right)+Q\left(\frac{\partial_{u}\phi}{\phi\ell}h+\frac{h}{h-\varepsilon_{0}}(n-n\varepsilon_{0}K^{\frac{1}{n}})\right).
\end{split}
\end{equation*}
By Lemma \ref{lemmaC0}, together with the continuity of $\phi$ and $\partial_{u}\phi$, we conclude that both $\phi$ and $\partial_{u}\phi$ are bounded on the bounded closed interval $\mathcal{C}_{\theta}\times I_{[0,T)}$, and $\ell$ is
a bounded function of $\xi$ on $\mathcal{C}_{\theta}$. Consequently,
\begin{equation*}
\partial_{t}Q\leq Q^{2}\left(C_{1}-C_{2}K^{\frac{1}{n}}\right)+Q\left(C_{1}-C_{2}K^{\frac{1}{n}}\right),
\end{equation*}
where $C_{1}$ is a positive constant depending on $n$, $\theta$, $\min_{\mathcal{C}_{\theta}}f$, $\max_{\mathcal{C}_{\theta}}f$, $||h||_{C^{0}(\mathcal{C}_{\theta}\times[0,T))}$, $||\phi||_{C^{0}(\mathcal{C}_{\theta}\times I_{[0,T)})}$
and $||\phi||_{C^{1}(\mathcal{C}_{\theta}\times I_{[0,T)})}$, $C_{2}$ is a positive constant depending on $n$, $\theta$ and $||h||_{C^{0}(\mathcal{C}_{\theta}\times[0,T))}$. If $K$ is sufficiently large, then
\begin{equation*}
\partial_{t}Q<0.
\end{equation*}
This contradicts the assumption that $Q$ attains a maximum at $\xi_{0}$ at time $t$, unless $K$ is uniformly bounded from above. This establishes the desired upper bound for $K$.
\end{proof}

\begin{lemma}\label{lemmaC2-2}
Let $\theta\in(0,\frac{\pi}{2})$. Suppose $h$ is a smooth, strictly convex solution to flow \eqref{flow-2} on $\mathcal{C}_{\theta}\times[0,T)$, and $f$, $\phi$ are smooth, positive functions satisfying \eqref{condition1}, then
\begin{equation*}
K(\xi,t)\geq C,
\end{equation*}
where $C$ is a positive constant depending on $n$, $\theta$, $\min_{\mathcal{C}_{\theta}}f$, $\max_{\mathcal{C}_{\theta}}f$, $||h||_{C^{0}(\mathcal{C}_{\theta}\times[0,T))}$, $||h||_{C^{1}(\mathcal{C}_{\theta}\times[0,T))}$,
$||\phi||_{C^{0}(\mathcal{C}_{\theta}\times I_{[0,T)})}$, $||\phi||_{C^{1}(\mathcal{C}_{\theta}\times I_{[0,T)})}$, $||\phi||_{C^{2}(\mathcal{C}_{\theta}\times I_{[0,T)})}$ and $\Sigma_{0}$.
\end{lemma}
\begin{proof}
Consider the auxiliary function
\begin{equation*}
Q=\log(f^{-1}K^{-1})-N\log h,
\end{equation*}
where $N$ is a positive constant to be determined later.

We may assume that $\min K\ll 1$, otherwise, the proof is complete. Due to Lemma \ref{lemmaC0}, the lower bound of $K(\xi,t)$ follows from the upper bound of $Q(\xi,t)$. Hence we only need to derive the upper bound of $Q(\xi,t)$,
we also assume that $\max Q\gg 1$.

\textbf{Claim}: \emph{For $t>0$, we have $\nabla_{\mu}Q<0$ on $\partial\mathcal{C}_{\theta}$. }

For $t>0$,
\begin{equation}\label{lemmaC2-2eq1}
\begin{split}
h_{t\mu}
=&\nabla_{\mu}(-fhK\psi+h)\\
=&-\nabla_{\mu}fhK\psi-\cot\theta fhK\psi-fh\nabla_{\mu}K\psi-fhK\nabla_{\mu}\psi+\cot\theta h,
\end{split}
\end{equation}
and
\begin{equation}\label{lemmaC2-2eq2}
h_{t\mu}=\cot\theta h_{t}=\cot\theta(-fhK\psi+h).
\end{equation}
Combining \eqref{lemmaC2-2eq1} and \eqref{lemmaC2-2eq2}, we have
\begin{equation}\label{lemmaC2-2eqK}
\frac{\nabla_{\mu}K}{K}=-\frac{\nabla_{\mu}f}{f}-\frac{\nabla_{\mu}\psi}{\psi}.
\end{equation}
According to the definition of $\psi$ and \eqref{def-u}, then
\begin{equation*}
\nabla_{\mu}\psi(\xi,\frac{h}{\ell})=\partial_{\mu}\psi(\xi,\frac{h}{\ell})+\partial_{u}\psi(\xi,\frac{h}{\ell})\nabla_{\mu}\left(\frac{h}{\ell}\right),
\end{equation*}
where $\partial_{\mu}\psi(\xi,\frac{h}{\ell})$ denotes the directional derivative of $\psi$ with respect to its first variable $\xi$ in the direction of $\mu$. By \eqref{C0-bd1}, we obtain
\begin{equation}\label{def-nablaphi}
\frac{\nabla_{\mu}\psi}{\psi}=\frac{\partial_{\mu}\psi}{\psi}=-\frac{\partial_{\mu}\phi}{\phi}.
\end{equation}
Hence, for sufficiently large $N$
\begin{equation*}
\nabla_{\mu}Q=-\frac{\nabla_{\mu}f}{f}-\frac{\nabla_{\mu}K}{K}-N\frac{\nabla_{\mu}h}{h}=-\frac{\partial_{\mu}\phi}{\phi}-N\cot\theta<0,
\end{equation*}
where we used both $\phi$ and $\partial_{\mu}\phi$ are bounded on the bounded closed interval $\mathcal{C}_{\theta}\times I_{[0,T)}$.

Therefore, for any fixed $t\in[0,T)$, the maximum of $Q(\xi,t)$ is attained at some point $\xi_{0}\in\mathcal{C}_{\theta}\backslash\partial\mathcal{C}_{\theta}$.
Choose a local orthonormal frame $\{e_{i}\}_{i=1}^{n}$ such that $\{b_{ij}\}$ is diagonal at the point $\xi_{0}$. At the point $\xi_{0}$,
\begin{equation}\label{lemmaC2-2eq3}
\nabla_{i}Q=-\frac{\nabla_{i}(fK)}{fK}-N\frac{\nabla_{i}h}{h}=0,
\end{equation}
and
\begin{equation}\label{lemmaC2-2ieq4}
\nabla_{ii}^{2}Q=-\frac{\nabla_{ii}^{2}(fK)}{fK}+\frac{(\nabla_{i}(fK))^{2}}{(fK)^{2}}-N\frac{\nabla_{ii}^{2}h}{h}+N\frac{(\nabla_{i}h)^{2}}{h^{2}}\leq0.
\end{equation}
Substituting \eqref{lemmaC2-2eq3} into \eqref{lemmaC2-2ieq4}, we obtain
\begin{equation}\label{lemmaC2-2ieq5}
-\nabla_{ii}^{2}(fK)\leq NfK\frac{\nabla_{ii}^{2}h}{h}+(-N-N^{2})fK\frac{(\nabla_{i}h)^{2}}{h^{2}}.
\end{equation}
Thus,
\begin{equation}\label{lemmaC2-2eq6}
\partial_{t}Q=\frac{\partial_{t}\sigma_{n}}{\sigma_{n}}-N\frac{h_{t}}{h}=\sum\limits_{i}b^{ii}(h_{tii}+h_{t})-N\frac{h_{t}}{h},
\end{equation}
and
\begin{equation*}
\begin{split}
h_{tii}
=&~\nabla_{ii}(-fhK\psi+h)\\
=&-(fK)_{ii}h\psi-fKh_{ii}\psi-fKh\psi_{ii}\\
&-2(fK)_{i}h\psi_{i}-2(fK)_{i}h_{i}\psi-2fKh_{i}\psi_{i}+h_{ii}.
\end{split}
\end{equation*}
Combining the definition of $\psi$ with \eqref{def-u}, we have
\begin{equation}\label{lemmaC2-2eq7}
\nabla_{i}\psi(\xi,\frac{h}{\ell})=\partial_{i}\psi(\xi,\frac{h}{\ell})+\partial_{u}\psi(\xi,\frac{h}{\ell})\nabla_{i}\left(\frac{h}{\ell}\right),
\end{equation}
where $\partial_{i}\psi(\xi,u)$ denotes the partial derivative of $\psi$ with respect to the $i$-th component $\xi_{i}$ of its first variable $\xi$. Since $\psi=\phi^{-1}$, we obtain
\begin{equation*}
\nabla_{i}\psi(\xi,\frac{h}{\ell})=-\frac{1}{\phi^{2}}\left(\partial_{i}\phi+\partial_{u}\phi(\frac{h_{i}}{\ell}-h\frac{\ell_{i}}{\ell^{2}})\right),
\end{equation*}
and
\begin{equation*}
\begin{split}
\nabla_{ii}^{2}\psi(\xi,\frac{h}{\ell})=&~\partial_{ii}^{2}\psi+2\partial_{i}(\partial_{u}\psi)\nabla_{i}(\frac{h}{\ell})+\partial_{uu}^{2}\psi\left(\nabla_{i}(\frac{h}{\ell})\right)^{2}+\partial_{u}\psi\nabla_{ii}^{2}(\frac{h}{\ell})\\
=&-\frac{1}{\phi^{2}}\partial_{ii}^{2}\phi+2\frac{1}{\phi^{3}}(\partial_{i}\phi)^{2}-2\left(\frac{1}{\phi^{2}}\partial_{i}(\partial_{u}\phi)-2\frac{1}{\phi^{3}}\partial_{i}\phi\cdot\partial_{u}\phi\right)(\frac{h_{i}}{\ell}-h\frac{\ell_{i}}{\ell^{2}})\\
&-\left(\frac{1}{\phi^{2}}\partial_{uu}^{2}\phi-2\frac{1}{\phi^{3}}(\partial_{u}\phi)^{2}\right)(\frac{h_{i}}{\ell}-h\frac{\ell_{i}}{\ell^{2}})^{2}
-\frac{\partial_{u}\phi}{\phi^{2}}\left(\frac{h_{ii}}{\ell}-2\frac{h_{i}\ell_{i}}{\ell^{2}}-\frac{h\ell_{ii}}{\ell^{2}}+2\frac{h(\ell_{i})^{2}}{\ell^{3}}\right).
\end{split}
\end{equation*}
Using Lemmas \ref{lemmaC0} and \ref{lemmaC1}, along with the boundedness of $\phi$, $\partial_{i}\phi$, $\partial_{u}\phi$, $\partial^{2}_{ii}\phi$, $\partial^{2}_{iu}\phi$ and $\partial^{2}_{uu}\phi$ on $\mathcal{C}_{\theta}\times I_{[0,T)}$, and the fact that $\ell$, $\ell_{i}$ and $\ell_{ii}$ are bounded functions of $\xi$
on $\mathcal{C}_{\theta}$, we have
\begin{equation}\label{lemmaC2-2ieq8}
\begin{split}
\nabla_{ii}^{2}\psi\geq -C_{1}(1+(\nabla_{i}h)^{2})-\frac{\partial_{u}\phi}{\phi^{2}\ell}\nabla_{ii}^{2}h,
\end{split}
\end{equation}
where $C_{1}$ is a positive constant depending on $\theta$, $||h||_{C^{0}(\mathcal{C}_{\theta}\times[0,T))}$, $||h||_{C^{1}(\mathcal{C}_{\theta}\times[0,T))}$, $||\phi||_{C^{0}(\mathcal{C}_{\theta}\times I_{[0,T)})}$,
$||\phi||_{C^{1}(\mathcal{C}_{\theta}\times I_{[0,T)})}$ and $||\phi||_{C^{2}(\mathcal{C}_{\theta}\times I_{[0,T)})}$.

In view of \eqref{lemmaC2-2eq3}, \eqref{lemmaC2-2ieq5}, \eqref{lemmaC2-2eq7} and \eqref{lemmaC2-2ieq8}, for sufficient large $N$, it follows that
\begin{equation}\label{lemmaC2-2ieq9}
\begin{split}
h_{tii}
\leq&(1+(N-1)fK\psi)\nabla_{ii}^{2}h+(N-N^{2})fhK\psi\frac{(\nabla_{i}h)^{2}}{h^{2}}\\
&+2(N-1)fK\nabla_{i}h\nabla_{i}\psi-fKh\nabla_{ii}^{2}\psi\\
\leq&\left(1+(N-1)fK\phi^{-1}+fKh\frac{\partial_{u}\phi}{\phi^{2}\ell}\right)(b_{ii}-h)\\
&-2(N-1)fK\phi^{-2}\left(\partial_{i}\phi-\partial_{u}\phi \frac{h\ell_{i}}{\ell^{2}}\right)\nabla_{i}h+C_{1}fKh\\
&+\left((N-N^{2})fh^{-1}K\phi^{-1}-2(N-1)fK\frac{\partial_{u}\phi}{\phi^{2}\ell}+C_{1}fKh\right)(\nabla_{i}h)^{2}.\\
\end{split}
\end{equation}
Substituting \eqref{lemmaC2-2ieq9} into \eqref{lemmaC2-2eq6}, we deduce
\begin{equation*}
\begin{split}
\partial_{t}Q
\leq&\left(1+(N-1)fK\phi^{-1}+fKh\frac{\partial_{u}\phi}{\phi^{2}\ell}\right)(n-h\sum\limits_{i}b^{ii})\\
&+\left((N-N^{2})fh^{-1}K\phi^{-1}-2(N-1)fK\frac{\partial_{u}\phi}{\phi^{2}\ell}+C_{1}fKh\right)\sum\limits_{i}b^{ii}(\nabla_{i}h)^{2}\\
&-\left(2(N-1)fK\phi^{-2}(\partial_{i}\phi-\partial_{u}\phi \frac{h\ell_{i}}{\ell^{2}})\nabla_{i}h-C_{1}fKh\right)\sum\limits_{i}b^{ii}+(\sum\limits_{i}b^{ii}-\frac{N}{h})h_{t}\\
=&\left(n-N-(n-N-nN)fK\phi^{-1}+nfKh\frac{\partial_{u}\phi}{\phi^{2}\ell}\right)\\
&+\left((N-N^{2})fh^{-1}K\phi^{-1}-2(N-1)fK\frac{\partial_{u}\phi}{\phi^{2}\ell}+C_{1}fKh\right)\sum\limits_{i}b^{ii}(\nabla_{i}h)^{2}\\
&-\left(2(N-1)fK(\frac{\partial_{i}\phi}{\phi^{2}}-\partial_{u}\phi \frac{h\ell_{i}}{\phi^{2}\ell^{2}})\nabla_{i}h-(C_{1}-\frac{N}{\phi}-h\frac{\partial_{u}\phi}{\phi^{2}\ell})fKh\right)\sum\limits_{i}b^{ii}.
\end{split}
\end{equation*}
By Lemmas \ref{lemmaC0} and \ref{lemmaC1} again, combined with the boundedness of $\phi$, $\partial_{i}\phi$, $\partial_{u}\phi$ on $\mathcal{C}_{\theta}\times I_{[0,T)}$ and of $\ell$, $\ell_{i}$ on $\mathcal{C}_{\theta}$,
simplifying the above inequality yields
\begin{equation}\label{lemmaC2-2ieq10}
\begin{split}
\partial_{t}Q
\leq&~n-N-(n-N-nN)fK\phi^{-1}+C_{2}fK\\
&+\left((N-N^{2})h^{-1}\phi^{-1}|\nabla_{i}h|^{2}+2(N-1)C_{2}(|\nabla_{i}h|^{2}+|\nabla_{i}h|)+(C_{2}-\frac{N}{C_{2}})\right)fK\sum\limits_{i}b^{ii},
\end{split}
\end{equation}
where $C_{2}$ is a positive constant depending on $n$, $\theta$, $||h||_{C^{0}(\mathcal{C}_{\theta}\times[0,T))}$, $||h||_{C^{1}(\mathcal{C}_{\theta}\times[0,T))}$, $||\phi||_{C^{0}(\mathcal{C}_{\theta}\times I_{[0,T)})}$,
$||\phi||_{C^{1}(\mathcal{C}_{\theta}\times I_{[0,T)})}$ and $C_{1}$.

Since $h$ and $\phi$ are bounded positive functions, if $\max Q\rightarrow+\infty$ (i.e. $\min K\rightarrow 0$), then for sufficiently large $N$, the right-hand side of inequality \eqref{lemmaC2-2ieq10} would become strictly negative. This would imply that $Q$ could not tend to infinity, which contradicts the assumption. Hence, $K$ must remain uniformly bounded below by a positive constant.
\end{proof}

\begin{lemma}\label{lemmaC2-3}
Let $\theta\in(0,\frac{\pi}{2})$. Suppose $h$ is a smooth, strictly convex solution to flow \eqref{flow-2} on $\mathcal{C}_{\theta}\times[0,T)$, and $f$, $\phi$ are smooth, positive functions satisfying \eqref{condition1}, then
\begin{equation*}
\frac{1}{C}\leq\kappa_{i}(\xi,t)\leq C,
\end{equation*}
where $C$ is a positive constant depending on $n$, $\theta$, $\min_{\mathcal{C}_{\theta}}f$, $\max_{\mathcal{C}_{\theta}}f$, $||f||_{C^{1}(\mathcal{C}_{\theta})}$, $||f||_{C^{2}(\mathcal{C}_{\theta})}$,
$||h||_{C^{0}(\mathcal{C}_{\theta}\times[0,T))}$, $||h||_{C^{1}(\mathcal{C}_{\theta}\times[0,T))}$, $||\phi||_{C^{0}(\mathcal{C}_{\theta}\times I_{[0,T)})}$, $||\phi||_{C^{1}(\mathcal{C}_{\theta}\times I_{[0,T)})}$, $||\phi||_{C^{2}(\mathcal{C}_{\theta}\times I_{[0,T)})}$ and $\Sigma_{0}$.
\end{lemma}
\begin{proof}
Since the Gauss curvature $K$ of $\Sigma_{t}$ is uniformly bounded from above, it suffices to prove the lower bound of the principal curvatures of $\Sigma_{t}$, which is equivalent to the upper bound of the principal radii of curvature of $\Sigma_{t}$.
To this end, we apply the maximum principle to the auxiliary function
\begin{equation*}
Q=\sigma_{1}+\frac{A}{2}|\nabla h|^{2},
\end{equation*}
where $A>0$ is a constant to be specified later, and $\sigma_{1}:=\sigma_{1}(b_{ij}(\xi,t))$.

\textbf{Claim}: \emph{Fix $A>0$ and suppose $t>0$. Then $\nabla_{\mu}Q<0$ at any point on $\partial\mathcal{C}_{\theta}$, or we can directly obtained a uniform upper bound for the principal curvatures, where $A$ is sufficiently large.}

Choose a local orthonormal frame $\{e_{i}\}_{i=1}^{n}$ such that $e_{n}=\mu$ and $\{b_{ij}\}$ is diagonal at the point $\xi_{0}\in\partial\mathcal{C}_{\theta}$. At the point $\xi_{0}$, we have
\begin{equation}\label{lemmaC2-3eq1}
\begin{split}
\nabla_{\mu}Q
=&\sum\limits_{j}\nabla_{\mu}b_{jj}+A\sum\limits_{j}h_{j}h_{j\mu}\\
=&\cot\theta\sum\limits_{\beta=1}^{n-1}(b_{\mu\mu}-b_{\beta\beta})+\nabla_{\mu}b_{\mu\mu}+A\cot\theta hh_{\mu\mu},
\end{split}
\end{equation}
where we used (see the proof of Lemma 3.3 in \cite{MWW})
\begin{equation*}
h_{\mu\beta}=0\quad \mathrm{and} \quad \nabla_{\mu}b_{\beta\beta}=\cot\theta (b_{\mu\mu}-b_{\beta\beta})
\end{equation*}
for all $1\leq\beta\leq n-1$. On the other hand, we have
\begin{equation*}
\begin{split}
\frac{\nabla_{\mu}K}{K}
=&-\sigma_{n}^{-1}\nabla_{\mu}\sigma_{n}\\
=&-\sigma_{n}^{-1}\left(\sum\limits_{\beta=1}^{n-1}\sigma_{n}^{\beta\beta}\nabla_{\mu}b_{\beta\beta}+\sigma_{n}^{\mu\mu}\nabla_{\mu}b_{\mu\mu}\right)\\
=&-\left(\cot\theta\sum\limits_{\beta=1}^{n-1}b^{\beta\beta}(b_{\mu\mu}-b_{\beta\beta})+b^{\mu\mu}\nabla_{\mu}b_{\mu\mu}\right)\\
=&-\cot\theta b_{\mu\mu}\sum\limits_{\beta=1}^{n-1}b^{\beta\beta}+(n-1)\cot\theta-b^{\mu\mu}\nabla_{\mu}b_{\mu\mu}.
\end{split}
\end{equation*}
Combing this with \eqref{lemmaC2-2eqK}, we obtain
\begin{equation}\label{lemmaC2-3eq2}
\nabla_{\mu}b_{\mu\mu}=\left(\frac{\nabla_{\mu}f}{f}+\frac{\nabla_{\mu}\psi}{\psi}+(n-1)\cot\theta\right)b_{\mu\mu}-\cot\theta b_{\mu\mu}^{2}\sum\limits_{\beta=1}^{n-1}b^{\beta\beta}.
\end{equation}
Substituting \eqref{lemmaC2-3eq2} into \eqref{lemmaC2-3eq1}, we find
\begin{equation*}
\begin{split}
\nabla_{\mu}Q
=&\cot\theta(nb_{\mu\mu}-\sigma_{1})+\left(\frac{\nabla_{\mu}f}{f}+\frac{\nabla_{\mu}\psi}{\psi}+(n-1)\cot\theta\right)b_{\mu\mu}\\
&-\cot\theta b_{\mu\mu}^{2}\sum\limits_{\beta=1}^{n-1}b^{\beta\beta}+A\cot\theta h(b_{\mu\mu}-h)\\
\leq&-\cot\theta\sigma_{1}+\left(\frac{\nabla_{\mu}f}{f}+\frac{\nabla_{\mu}\psi}{\psi}+(2n-1+Ah)\cot\theta\right)b_{\mu\mu}\\
&-(n-1)\cot\theta K^{\frac{1}{n-1}}b_{\mu\mu}^{2+\frac{1}{n-1}},
\end{split}
\end{equation*}
where we used
\begin{equation*}
\sum\limits_{\beta=1}^{n-1}b^{\beta\beta}\geq(n-1)\left(\prod_{\beta=1}^{n-1}b^{\beta\beta}\right)^{\frac{1}{n-1}}=(n-1)K^{\frac{1}{n-1}}b_{\mu\mu}^{\frac{1}{n-1}}.
\end{equation*}
In view of Lemma \ref{lemmaC0}, \ref{lemmaC2-2} and \eqref{def-nablaphi}, we have
\begin{equation*}
\nabla_{\mu}Q\leq\cot\theta\left(-\sigma_{1}+C_{1}b_{\mu\mu}-C_{2}b_{\mu\mu}^{2+\frac{1}{n-1}}\right),
\end{equation*}
where $C_{1}$, $C_{2}$ are positive constants, with $C_{1}$ depending on $n$, $A$, $\theta$, $\min_{C^{0}}f$, $||f||_{C^{1}(\mathcal{C}_{\theta})}$,
$||h||_{C^{0}(\mathcal{C}_{\theta}\times[0,T))}$, $||h||_{C^{1}(\mathcal{C}_{\theta}\times[0,T))}$, $||\phi||_{C^{0}(\mathcal{C}_{\theta}\times I_{[0,T)})}$, $||\phi||_{C^{1}(\mathcal{C}_{\theta}\times I_{[0,T)})}$,
$C_{2}$ depending on $n$ and $K$.

\textbf{Case $1$}. $Q\geq C_{1}b_{\mu\mu}+\frac{A}{2}|\nabla h|^{2}$ (i.e. $\sigma_{1}\geq C_{1}b_{\mu\mu}$) at $\xi_{0}$, then it is clear that $\nabla_{\mu}Q<0$.

\textbf{Case $2$}. $Q\leq C_{1}b_{\mu\mu}+\frac{A}{2}|\nabla h|^{2}$
 (i.e. $\sigma_{1}\leq C_{1}b_{\mu\mu}$) at $\xi_{0}$, it follows that
\begin{equation*}
\begin{split}
\nabla_{\mu}Q\leq&\cot\theta b_{\mu\mu}\left(C_{1}-C_{2}b_{\mu\mu}^{\frac{n}{n-1}}\right)\\
\leq&\cot\theta b_{\mu\mu}\left(C_{1}-C_{2}C_{1}^{-\frac{n}{n-1}}\sigma_{1}^{\frac{n}{n-1}}\right).
\end{split}
\end{equation*}
If $\sigma_{1}> C_{1}^{2-\frac{1}{n}}C_{2}^{\frac{1-n}{n}}$, then $\nabla_{\mu}Q$ is strictly negative. If $\sigma_{1}\leq C_{1}^{2-\frac{1}{n}}C_{2}^{\frac{1-n}{n}}$, by Lemma \ref{lemmaC1}, we can find a time-independent positive constant upper bound for $Q$ 
at any point on $\partial\mathcal{C}_{\theta}$. Consequently, if the maximum of $Q$ attained at some point $\xi_{0}\in\partial\mathcal{C}_{\theta}$, we can obtain the upper bound of the principal curvatures of $\Sigma_{t}$ directly.

In both cases, we have either shown $\nabla_{\mu}Q<0$ at a potential boundary maximum, or directly obtained a uniform upper bound for the principal curvatures. This establishes the claim.

Therefore, for any fixed $t\in[0\times T)$ and sufficient large $A$, according to the above claim, if the maximum of $Q$ is attained on $\xi_{0}\in\partial\mathcal{C}_{\theta}$, we have either gotten a contradiction or directly obtained a time-independent positive constant upper bound for the principal curvatures. Hence, it suffices to consider the case the maximum of $Q$ attained at some point $\xi_{0}\in\mathcal{C}_{\theta}\setminus\partial\mathcal{C}_{\theta}$. Choose a local orthonormal frame $\{e_{i}\}_{i=1}^{n}$ such that $\{b_{ij}(\xi_{0},t)\}$ is diagonal, hence $\{h_{ij}(\xi_{0},t)\}$
is also diagonal. At $\xi_{0}$, we calculate
\begin{equation*}
0=\nabla_{i}Q=\sum\limits_{j}\nabla_{i}b_{jj}+Ah_{i}h_{ii},
\end{equation*}
and
\begin{equation}\label{lemmaC2-3ieq3}
0\geq\nabla_{ii}^{2}Q=\sum\limits_{j}\nabla_{ii}^{2}b_{jj}+A\left(\sum\limits_{j}h_{j}h_{jii}+h_{ii}^{2}\right),
\end{equation}
where $h_{jii}:=\nabla_{i}h_{ji}$. Moreover, we have
\begin{equation}\label{lemmaC2-3eq4}
\partial_{t}Q=\sum\limits_{j}\partial_{t}b_{jj}+A\sum\limits_{j}h_{j}h_{jt}.
\end{equation}
On the other hand,
\begin{equation}\label{lemmaC2-3eq5}
\begin{split}
\log(h-h_{t})=&\log(fKh\psi)\\
=&-\log\det(\nabla^{2}h+h\emph{I})+P,
\end{split}
\end{equation}
where $P:=\log (fh\psi)$. Taking the covariant derivative of \eqref{lemmaC2-3eq5} in the $e_{j}$ direction, combining with \eqref{ricci-eq1}, it follows that
\begin{equation}\label{lemmaC2-3eq6}
\begin{split}
\frac{h_{j}-h_{jt}}{h-h_{t}}
=&-\sum\limits_{i,k}b^{ik}\nabla_{j}b_{ik}+P_{j}\\
=&-\sum\limits_{i}b^{ii}(h_{jii}+h_{i}\delta_{ij})+P_{j},
\end{split}
\end{equation}
where we used that $\nabla b$ is fully symmetric. We also have
\begin{equation*}
\frac{h_{jj}-h_{jjt}}{h-h_{t}}-\frac{(h_{j}-h_{jt})^{2}}{(h-h_{t})^{2}}=-\sum\limits_{i}b^{ii}\nabla_{jj}^{2}b_{ii}+\sum\limits_{i,k}b^{ii}b^{kk}(\nabla_{j}b_{ik})^{2}+P_{jj}.
\end{equation*}
Using the commutator formula \eqref{ricci-eq2}, we find
\begin{equation}\label{lemmaC2-3eq7}
\begin{split}
\frac{h_{jj}-h_{jjt}}{h-h_{t}}
=&\frac{(h_{j}-h_{jt})^{2}}{(h-h_{t})^{2}}-\sum\limits_{i}b^{ii}(\nabla_{ii}^{2}b_{jj}-b_{jj}+b_{ii})\\
&+\sum\limits_{i,k}b^{ii}b^{kk}(\nabla_{j}b_{ik})^{2}+P_{jj}\\
=&\frac{(h_{j}-h_{jt})^{2}}{(h-h_{t})^{2}}-\sum\limits_{i}b^{ii}\nabla_{ii}^{2}b_{jj}+b_{jj}\sum\limits_{i}b^{ii}-n\\
&+\sum\limits_{i,k}b^{ii}b^{kk}(\nabla_{j}b_{ik})^{2}+P_{jj}.
\end{split}
\end{equation}
Together with \eqref{lemmaC2-3ieq3}, \eqref{lemmaC2-3eq4} and \eqref{lemmaC2-3eq7}, we calculate
\begin{equation}\label{lemmaC2-3eq8}
\begin{split}
\frac{\partial_{t}Q}{h-h_{t}}
=&\sum\limits_{j}\frac{(h_{jjt}-h_{jj})+(h_{jj}+h)-h+h_{t}}{h-h_{t}}+A\sum\limits_{j}\frac{h_{j}h_{jt}}{h-h_{t}}\\
=&-\sum\limits_{j}\frac{(h_{j}-h_{jt})^{2}}{(h-h_{t})^{2}}+\sum\limits_{i,j}b^{ii}\nabla_{ii}^{2}b_{jj}+n^{2}-\sum\limits_{i,j,k}b^{ii}b^{kk}(\nabla_{j}b_{ik})^{2}\\
&-\triangle P-\sum\limits_{j}b_{jj}\sum\limits_{i}b^{ii}+\sum\limits_{j}\frac{b_{jj}}{h-h_{t}}-n+A\sum\limits_{j}\frac{h_{j}h_{jt}}{h-h_{t}}\\
\leq&-A\sum\limits_{i,j}b^{ii}h_{j}h_{jii}-A\sum\limits_{i}b^{ii}h_{ii}^{2}+n(n-1)-\triangle P\\
&-\sigma_{1}\sum\limits_{i}b^{ii}+\frac{\sigma_{1}}{h-h_{t}}+A\sum\limits_{j}\frac{h_{j}h_{jt}}{h-h_{t}}.
\end{split}
\end{equation}
Multiplying both sides of \eqref{lemmaC2-3eq6} by $h_{j}$ and summing over $j$, we find
\begin{equation*}
\frac{|\nabla h|^{2}-\sum_{j}h_{j}h_{jt}}{h-h_{t}}=-\sum\limits_{i,j}b^{ii}h_{j}h_{jii}-\sum\limits_{i}b^{ii}h_{i}^{2}+\sum\limits_{j}h_{j}P_{j}.
\end{equation*}
Substituting this into \eqref{lemmaC2-3eq8} yields
\begin{equation}\label{lemmaC2-3eq9}
\begin{split}
\frac{\partial_{t}Q}{h-h_{t}}
\leq&~n(n-1)-\triangle P-\sigma_{1}\sum\limits_{i}b^{ii}+\frac{\sigma_{1}}{h-h_{t}}\\
&+A\left(\frac{|\nabla h|^{2}}{h-h_{t}}-\sum\limits_{j}h_{j}P_{j}+\sum\limits_{i}b^{ii}(h_{i}^{2}-h_{ii}^{2})\right)\\
=&~n(n-1)-\triangle P-\sigma_{1}\sum\limits_{i}b^{ii}+\frac{\sigma_{1}}{h-h_{t}}\\
&+A\frac{|\nabla h|^{2}}{h-h_{t}}-A\sum\limits_{j}h_{j}P_{j}+A\sum\limits_{i}b^{ii}h_{i}^{2}\\
&-A\sigma_{1}+2nAh-Ah^{2}\sum\limits_{i}b^{ii}.
\end{split}
\end{equation}
We also have
\begin{equation*}
P_{j}=\frac{f_{j}}{f}+\frac{h_{j}}{h}+\frac{\psi_{j}}{\psi}=\frac{f_{j}}{f}+\frac{h_{j}}{h}-\frac{1}{\phi}(\partial_{j}\phi+\partial_{u}\phi\frac{h_{j}}{\ell}-\partial_{u}\phi \frac{h\ell_{j}}{\ell^{2}}).
\end{equation*}
and
\begin{equation*}
\begin{split}
P_{jj}
=&\frac{f_{jj}}{f}-\frac{(f_{j})^{2}}{f^{2}}+\frac{h_{jj}}{h}-\frac{(h_{j})^{2}}{h^{2}}+\frac{\psi_{jj}}{\psi}-\frac{(\psi_{j})^{2}}{\psi^{2}}\\
=&\frac{f_{jj}}{f}-\frac{(f_{j})^{2}}{f^{2}}+\frac{h_{jj}}{h}-\frac{(h_{j})^{2}}{h^{2}}-\frac{1}{\phi}\partial_{jj}^{2}\phi+2\frac{1}{\phi^{2}}(\partial_{j}\phi)^{2}\\
&-(2\frac{\partial_{j}(\partial_{u}\phi)}{\phi}-4\frac{\partial_{j}\phi\cdot\partial_{u}\phi}{\phi^{2}})(\frac{h_{j}}{\ell}-\frac{h\ell_{j}}{\ell^{2}})
-(\frac{\partial_{uu}^{2}\phi}{\phi}-2\frac{(\partial_{u}\phi)^{2}}{\phi^{2}})(\frac{h_{j}}{\ell}-\frac{h\ell_{j}}{\ell^{2}})^{2}\\
&-\frac{\partial_{u}\phi}{\phi}(\frac{h_{jj}}{\ell}-2\frac{h_{j}\ell_{j}}{\ell^{2}}-\frac{h\ell_{jj}}{\ell^{2}}+2\frac{h(\ell_{j})^{2}}{\ell^{3}})-(\frac{\partial_{j}\phi}{\phi}+\partial_{u}\phi\frac{h_{j}}{\phi\ell}-\partial_{u}\phi\frac{h\ell_{j}}{\phi\ell^{2}})^{2}.
\end{split}
\end{equation*}
Using Lemmas \ref{lemmaC0} and \ref{lemmaC1}, along with the boundedness of $\phi$, $\partial_{i}\phi$, $\partial_{u}\phi$, $\partial^{2}_{ii}\phi$, $\partial^{2}_{iu}\phi$ and $\partial^{2}_{uu}\phi$ on $\mathcal{C}_{\theta}\times I_{[0,T)}$, and the fact that $f$, $\ell$, $\ell_{i}$ and $\ell_{ii}$ are bounded functions of $\xi$, it follows that
\begin{equation}\label{lemmaC2-3eq9}
-C_{3}\leq P_{j}\leq C_{3},\quad P_{jj}\geq -C_{4}h_{jj}-C_{4},
\end{equation}
where $C_{3}$ is a positive constant depending on $n$, $\theta$, $\min_{C^{0}}f$, $\max_{C^{0}}f$, $||f||_{C^{1}(\mathcal{C}_{\theta})}$, $||h||_{C^{0}(\mathcal{C}_{\theta}\times[0,T))}$, $||h||_{C^{1}(\mathcal{C}_{\theta}\times[0,T))}$,  $||\phi||_{C^{0}(\mathcal{C}_{\theta}\times I_{[0,T)})}$ and $||\phi||_{C^{1}(\mathcal{C}_{\theta}\times I_{[0,T)})}$,
with $C_{4}$ is a positive constant depending on $n$, $\theta$, $\min_{C^{0}}f$, $\max_{C^{0}}f$, $||f||_{C^{1}(\mathcal{C}_{\theta})}$, $||f||_{C^{2}(\mathcal{C}_{\theta})}$, $||h||_{C^{0}(\mathcal{C}_{\theta}\times[0,T))}$, $||h||_{C^{1}(\mathcal{C}_{\theta}\times[0,T))}$, $||\phi||_{C^{0}(\mathcal{C}_{\theta}\times I_{[0,T)})}$, $||\phi||_{C^{1}(\mathcal{C}_{\theta}\times I_{[0,T)})}$, $||\phi||_{C^{2}(\mathcal{C}_{\theta}\times I_{[0,T)})}$.

Therefore, in view of \eqref{lemmaC2-3eq9}, Lemmas \ref{lemmaC0} and \ref{lemmaC1}, we obtain
\begin{equation}\label{lemmaC2-3eq10}
\begin{split}
\frac{\partial_{t}Q}{h-h_{t}}
\leq C_{5}(1+A)+(C_{5}-A)\sigma_{1}+(A|\nabla h|^{2}-\sigma_{1})\sum\limits_{i}b^{ii},
\end{split}
\end{equation}
where $C_{5}$ is a positive constant depending on $n$, $\theta$, $\min_{C^{0}}f$, $\max_{C^{0}}f$, $||h||_{C^{0}(\mathcal{C}_{\theta}\times[0,T))}$, $||h||_{C^{1}(\mathcal{C}_{\theta}\times[0,T))}$, $||\phi||_{C^{0}(\mathcal{C}_{\theta}\times I_{[0,T)})}$, $||\phi||_{C^{1}(\mathcal{C}_{\theta}\times I_{[0,T)})}$, $C_{3}$ and $C_{4}$.

Since $h-h_{t}$ is positive and uniformly bounded above, if $\max Q\rightarrow\infty$ (i.e. $\max\sigma_1\rightarrow\infty$), then for sufficiently large $A$, the right-hand side of inequality \eqref{lemmaC2-3eq10} would become strictly negative. This would imply that $Q$ could not tend to infinity, which contradicts to the assumption. Therefore $\sigma_{1}$ must remain uniformly bounded above by a positive constant, we can obtain the upper bound of the principal curvatures of $\Sigma_{t}$ directly, this completes the proof.
\end{proof}

\section{Existence of the solutions to the equation}  \label{S4}

In this section, we first complete the proof of Theorem \ref{maintherorem2}, which amounts to proving that the support function $h_{\infty}$ of $\Sigma_{\infty}$ satisfies the following equation:
\begin{equation}
\left\{
\begin{aligned}
&\phi\left(\xi,\frac{h}{\ell}\right)\det(h_{ij}+h\delta_{ij})=f,\qquad&&~\mathrm{in}~\mathcal{C}_{\theta},\\
&\\
&\nabla_{\mu}h=\cot\theta h, \qquad&&~\mathrm{on}~ \partial \mathcal{C}_{\theta}.\\
\end{aligned}
\right.
\end{equation}
To achieve this, we construct a monotone quantity.

Define
\begin{equation*}
\Phi(\xi,t)=\int^{t}_0\frac{1}{\phi\left(\xi,s\right)}ds.
\end{equation*}
For the flow \eqref{flow-2}, we define the functional
\begin{equation*}
J(\Sigma_{t})=\int_{\mathcal{C}_{\theta}}f(\xi)\Phi\left(\xi,\frac{h}{\ell}\right)\ell d\xi-V(\widehat{\Sigma_{t}}),
\end{equation*}
where $\ell=\sin^{2}\theta+\cos\theta\langle \xi,e\rangle$ and $V(\widehat{\Sigma_{t}})$ denote the volume of the convex body $\widehat{\Sigma_{t}}$.
Now, we show that the functional $J(\Sigma_{t})$ is non-increasing along the flow \eqref{flow-2}.

\begin{lemma}\label{mon-decrease}
The functional $J(\Sigma_{t})$ is non-increasing along the flow \eqref{flow-2}. That is
\begin{equation*}
\frac{d}{dt}J(\Sigma_{t})\leq 0,
\end{equation*}
and the equality holds if and only if $\Sigma_{t}$ satisfies \eqref{omeq-1}.
\end{lemma}
\begin{proof}
For $t>0$, by \cite[Theorem 2.7]{WWX24}, we obtain
\begin{equation*}
\frac{d}{dt}V(\widehat{\Sigma_{t}})=\int_{\mathcal{C}_{\theta}}\frac{\partial_{t}h}{K}d\xi.
\end{equation*}
Since
\begin{equation*}
\begin{split}
\frac{d}{dt}J(\Sigma_{t})=&\int_{\mathcal{C}_{\theta}}f(\xi)\frac{\partial_{t}h}{\phi\left(\xi,\frac{h}{\ell}\right)} d\xi-\int_{\mathcal{C}_{\theta}}\frac{\partial_{t}h}{K}d\xi\\
=&\int_{\mathcal{C}_{\theta}}\left(\frac{f(\xi)}{\phi\left(\xi,\frac{h}{\ell}\right)}-\frac{1}{K}\right)\frac{\partial h}{\partial t}d\xi\\
=&-\int_{\mathcal{C}_{\theta}}\frac{h}{K}\left(\frac{f(\xi)K}{\phi\left(\xi,\frac{h}{\ell}\right)}-1\right)^{2}d\xi\leq 0.
\end{split}
\end{equation*}
The equality holds if and only if
\begin{equation*}
\frac{f(\xi)K}{\phi\left(\xi,\frac{h}{\ell}\right)}=1,
\end{equation*}
which is just the first equation of \eqref{omeq-1}. The boundary condition follows naturally from the definition of the flow, this completes the proof.
\end{proof}

\begin{proof}[Proof of Theorem \ref{maintherorem2}]
In view of Lemmas \ref{lemmaC0}$\sim$\ref{lemmaC2-3}, we conclude that the solution $h$ of flow \eqref{flow-2} satisfies
\begin{equation*}
||h(\cdot,t)||_{C^{2}(\mathcal{C}_{\theta})}\leq C,
\end{equation*}
where the positive constant $C$ is independent of $t$. The theory of standard parabolic fully nonlinear equations with Neumann boundary conditions (see, e.g., \cite[Theorems 6.1, 6.4 and 6.5]{D88} and \cite[Theorem 14.23]{L96}) ensures that, for any
integer $k\geq2$ and $\alpha\in(0,1)$, the following holds
\begin{equation}\label{holder estimate}
||h(\cdot,t)||_{C^{k,\alpha}(\mathcal{C}_{\theta})}\leq C.
\end{equation}
Hence we obtain the long time existence and $C^{\infty}$-smoothness of solutions for the flow \eqref{flow-2}.
According to the definition of $J(\Sigma_{\tau})$, and recalling Lemmas \ref{lemmaC0}, \ref{lemmaC2-1}, \ref{lemmaC2-2}, \ref{lemmaC2-3} and \ref{mon-decrease}, it follows that
\begin{equation*}
|J(\Sigma_{0})-J(\Sigma_{\tau})|=|\int_{0}^{\tau}dt\int_{\mathcal{C}_{\theta}}\frac{h}{K}\left(\frac{f(\xi)K}{\phi\left(\xi,\frac{h}{\ell}\right)}-1\right)^{2}d\xi|<\infty.
\end{equation*}
By the monotonicity of $J(\Sigma_{t})$, for a number sequence $\{t_{j}\}_{j\geq1}$ with $t_{j}\rightarrow +\infty$, the limit $h(\xi,\infty):=\lim_{t_{j}\rightarrow\infty}h(\xi, t_{j})$ exists. Passing to the limit, we have
\begin{equation*}
\int_{\mathcal{C}_{\theta}}\frac{h_{\infty}}{K}\left(\frac{f(\xi)K}{\phi\left(\xi,\frac{h_{\infty}}{\ell}\right)}-1\right)^{2}d\xi=0,
\end{equation*}
where $h_{\infty}:=h(\xi,\infty)$ is the capillary support function of $\Sigma_{\infty}$. This implies that
\begin{equation*}
\frac{f(\xi)K}{\phi\left(\xi,\frac{h_{\infty}}{\ell}\right)}=1,
\end{equation*}
combing with the fact that $h_{\infty}$ satisfies the flow \eqref{flow-2}, which is just equation \eqref{omeq-1}. The proof of Theorem \ref{maintherorem2} is complete.
\end{proof}

\begin{proof}[Proof of Theorem \ref{maintherorem1} ]
The existence of a smooth, strictly convex solution to equation \eqref{omeq-1} is a direct consequence of Theorem \ref{maintherorem2}.

By Theorem \ref{maintherorem2}, starting from the smooth, strictly convex initial hypersurface $\Sigma_{0}$, the solution $h(\xi,t)$ to flow \eqref{flow-2} exists for all time and converges smoothly to a limit function $h_{\infty}(\xi)$ (as $t\rightarrow\infty$).
And the proof of Theorem \ref{maintherorem2} implies that this limit function $h_{\infty}$ satisfies
\begin{equation*}
f(\xi)K=\phi(\xi, \frac{h_\infty}{\ell}),
\end{equation*}
which is the first equation in \eqref{omeq-1}. Since $h_{\infty}$ satisfies the flow \eqref{flow-2}, the boundary condition is naturally satisfied. According to Theorem \ref{maintherorem2}, the strict convexity and smoothness of the solution  $h_{\infty}$ are preserved.

Therefore, $h_{\infty}$ is the desired smooth, strictly convex solution to \eqref{omeq-1}. This completes the proof of Theorem \ref{maintherorem1}.
\end{proof}

\end{document}